\documentclass[12pt]{amsart}

\usepackage{amssymb, amsmath, amsthm}
\usepackage[margin=1in]{geometry}
\usepackage[dvipdfmx]{graphicx, xcolor} 
\usepackage{tikz}

\allowdisplaybreaks

\numberwithin{equation}{section}

\theoremstyle{plain}
\newtheorem{thm}{Theorem}[section]
\newtheorem{prop}[thm]{Proposition}
\newtheorem{lem}[thm]{Lemma}

\newtheorem*{referthmA}{Theorem A}
\newtheorem*{referthmB}{Theorem B}
\newtheorem*{referthmC}{Theorem C}

\theoremstyle{definition}
\newtheorem{defn}[thm]{Definition}

\newtheorem{notation}[thm]{Notation}

\newcommand{\ichi}{\mathbf{1}}

\newcommand{\N}{\mathbb{N}}
\newcommand{\R}{\mathbb{R}}

\newcommand{\Z}{\mathbb{Z}}

\newcommand{\calK}{\mathcal{K}}

\newcommand{\calS}{\mathcal{S}}

\newcommand{\supp}{\mathrm{supp}\, }

\newcommand{\RomI}{\mathrm{I}}
\newcommand{\II}{\mathrm{I}\hspace{-0.5pt}\mathrm{I}}

\begin{document}

\title[Bilinear Fourier integral operators]
{Estimates for a certain bilinear Fourier integral operator} 

\author[T. Kato]{Tomoya Kato}
\author[A. Miyachi]{Akihiko Miyachi}
\author[N. Tomita]{Naohito Tomita}

\address[T. Kato]
{Division of Pure and Applied Science, 
Faculty of Science and Technology, Gunma University, 
Kiryu, Gunma 376-8515, Japan}

\address[A. Miyachi]
{Department of Mathematics, 
Tokyo Woman's Christian University, 
Zempukuji, Suginami-ku, Tokyo 167-8585, Japan}

\address[N. Tomita]
{Department of Mathematics, 
Graduate School of Science, Osaka University, 
Toyonaka, Osaka 560-0043, Japan}

\email[T. Kato]{t.katou@gunma-u.ac.jp}
\email[A. Miyachi]{miyachi@lab.twcu.ac.jp}
\email[N. Tomita]{tomita@math.sci.osaka-u.ac.jp}

\date{\today}

\keywords
{Bilinear Fourier integral operators}

\thanks
{This work was supported by JSPS KAKENHI, 
Grant Numbers 
20K14339 (Kato), 
20H01815 (Miyachi), and 
20K03700 (Tomita).}

\subjclass[2020]
{42B15, 
42B20}

\begin{abstract}
In this paper,
we consider the boundedness
from $H^{1} \times L^{\infty}$ to $L^{1}$ of
bilinear Fourier integral operators
with non-degenerate phase functions and 
amplitudes in $BS_{1,0}^{-n/2}$.
Our result gives an improvement of
Rodr\'iguez-L\'opez, Rule, and Staubach's result.
\end{abstract}

\maketitle

\section{Introduction}\label{intro}

The (linear) Fourier integral operator $T_{\sigma}^{\phi}$ 
treated in this paper
is the following form:
\begin{equation*} 
T_{\sigma}^{\phi} f (x)
= \int_{\R^n} e^{i \phi(x, \xi)} 
\sigma(x, \xi) \widehat{f} (\xi) \,d\xi ,
\quad x \in \R^n,
\end{equation*}
for
$f$ in the Schwartz class $\calS(\R^n)$,
where $\widehat{f}$ denotes the Fourier transform.
The function $\phi$ is called the {\it phase function},
usually assumed homogeneity and non-degeneracy conditions,
and the function $\sigma$ is called the {\it amplitude},
usually taken to be in some H\"ormander class
(see Definitions \ref{def-ND} and \ref{def-Sm10} below).
For the boundedness 
of the operator $T_{\sigma}^{\phi}$, 
we use the following terminology. 
Let $X$ and $Y$ be function spaces on $\R^n$ 
equipped with quasi-norms or seminorms 
$\|\cdot\|_{X}$ and $\|\cdot\|_{Y}$, respectively.
If there exists a constant $C$ such that
\begin{equation}\label{bdd-XY}
\|T_{\sigma}^{\phi} f\|_{Y}
\le C \|f\|_{X} 
\end{equation}
holds for all 
$f \in \calS \cap X$, 
then we say that 
$T_{\sigma}^{\phi}$ is bounded from 
$X$ to $Y$. 
The smallest constant $C$ of 
\eqref{bdd-XY} 
is denoted by 
$\|T_{\sigma}^{\phi}\|_{X \to Y}$.

In order to illustrate boundedness results
of the operator $T_{\sigma}^{\phi}$,
we recall the following definitions 
for phase functions and amplitudes.

\begin{defn} \label{def-ND} 
A function 
$\phi (x,\xi) \in C^{\infty} (\R^{n} \times (\R^{n} \setminus \{0\}))$ 
is called a non-degenerate phase function,
if it is real-valued,
positively homogeneous of degree one 
in the variable $\xi$, {\it i.e.},
\[
\phi (x, t \xi) = t \phi (x,\xi), 
\quad
t > 0, 
\;\; x \in \R^n , 
\;\; \xi \in \R^n \setminus\{0\} ,
\]
and satisfies the {\it non-degeneracy condition}
\[
\det \Big( \frac{\partial^2 \phi (x,\xi)}{\partial{x} \partial{\xi}} \Big) \neq 0
\]
for $x$ in the spatial support of the amplitude and 
for $\xi \in \R^n \setminus\{0\}$.
\end{defn}

\begin{defn} \label{def-Sm10} 
Let $m\in \R$ and $N \in \N_{0} = \N \cup \{0\}$. 
The class $S^{m}_{1,0} (N)$ 
is defined to be the set of all 
$C^{\infty}$ functions $\sigma = \sigma (x, \xi)$ 
on $\R^{2n}$ 
that satisfy the estimate 
\begin{equation*} 
\left|
\partial^{\alpha}_{x} 
\partial^{\beta}_{\xi} 
\sigma (x, \xi)
\right| 
\le 
C_{\alpha, \beta} 
\left( 1+ |\xi|\right)^{m-|\beta|}
\end{equation*}
for all multi-indices satisfying
$|\alpha|, |\beta| \le N$.
We define
\[
\| \sigma \|_{ S_{1,0}^{m} (N) }
= \max_{ |\alpha|, |\beta| \le N }
\left( \sup_{x, \xi \in \R^n}
\left( 1+ |\xi| \right)^{-(m-|\beta|)}
\big| 
\partial^{\alpha}_{x} 
\partial^{\beta}_{\xi} 
\sigma (x, \xi)
\big|
\right) .
\]
\end{defn}

Now, we recall the following theorem.

\begin{referthmA}
Let $\phi$ be a non-degenerate phase function 
and an amplitude $\sigma(x, \xi)$
be compactly supported in the spatial variable $x$.
Then, if $1 < p < \infty$ and $m = -(n-1) |1/p-1/2|$,
or if $1 \le p \le \infty$ and $m<-(n-1) |1/p-1/2|$,
there exist a positive constant $C$ 
and a positive integer $N$
such that 
\begin{align*}
\| T_{\sigma}^{\phi} \|_{L^p \to L^p}
\le C \| \sigma \|_{ S_{1,0}^{m} (N) }.
\end{align*}
\end{referthmA}

The first assertion was proved by
H\"ormander \cite{Hor1971} for $p=2$ and
Seeger--Sogge--Stein \cite{SSS} for $1 < p < \infty$,
and the second by
Dos Santos Ferreira--Staubach 
\cite[Theorems 2.1, 2.2, and 2.16]{DS}. 
It is worth mentioning that 
the number $-(n-1) |1/p-1/2|$
is optimal
(see \cite{M-wave}).

Based on the above result,
we shall consider a bilinear analogue.
The bilinear Fourier integral operator 
$T_{\sigma}^{\phi_1, \phi_2}$
considered in this paper
is defined by,
for non-degenerate phase functions $\phi_1$ and $\phi_2$
and for an amplitude $\sigma \in L^{\infty} (\R^{3n})$,
\begin{equation*}
T_{\sigma}^{\phi_1, \phi_2} (f, g) (x)
= 
\iint_{\R^n \times \R^n} 
e^{i \phi_1(x, \xi)+ i \phi_2 (x, \eta)} 
\sigma(x, \xi, \eta)
\widehat{f} (\xi) \widehat{g} (\eta) 
\,d\xi d\eta ,
\quad x \in \R^n,
\end{equation*}
for $f, g \in \calS(\R^n)$.
If $X$, $Y$, and $Z$ are function spaces on $\R^n$ 
equipped with quasi-norms or seminorms 
$\|\cdot\|_{X}$, $\|\cdot\|_{Y}$, and 
$\|\cdot\|_{Z}$, respectively,
and if there exists a constant $C$
such that
\begin{equation}\label{bdd-XYZ}
\|T_{\sigma}^{\phi_1, \phi_2} (f, g)\|_{Z}
\le C \|f\|_{X}\|g\|_{Y} 
\end{equation}
holds for all 
$f \in \calS \cap X$ and $g \in \calS \cap Y$, 
then we say that 
$T_{\sigma}^{\phi_1, \phi_2}$ is bounded from 
$X \times Y$ to $Z$. 
The smallest constant $C$ of 
\eqref{bdd-XYZ} 
is denoted by 
$\|T_{\sigma}^{\phi_1, \phi_2}\|_{X \times Y \to Z}$. 
For the functions spaces $X$, $Y$, and $Z$,
we consider the Lebesgue space $L^p$,
the Hardy space $H^p$,
the local Hardy space $h^p$, 
and the space $BMO$.
We note that
$H^p = h^p = L^p$ if $1<p\le \infty$ and 
$H^1 \hookrightarrow h^1 \hookrightarrow L^1$. 
For $H^p$ and $BMO$, see, 
{\it e.g.,} \cite[Chapters III and IV]{S}
and, for $h^p$, see \cite{Goldberg}.  
In the special case when
$\phi_{i}(x,\xi) = x \cdot \xi$, $i=1,2$,
the bilinear Fourier integral operator is called 
a bilinear pseudo-differential operator,
which is the operator originated by Coifman--Meyer \cite{CM},
and now vastly wide study has been done 
(see, {\it e.g.}, \cite{BBMNT, BT, GT}).

Now, 
in order to state the boundedness result 
for the bilinear operator $T_{\sigma}^{\phi_1, \phi_2}$,
we define the bilinear H\"ormander class as follows.

\begin{defn}\label{def-BSm10} 
For $m\in \R$, 
the class $BS^{m}_{1,0}$ is defined to be the set of all 
$C^{\infty}$ functions $\sigma = \sigma (x, \xi, \eta)$ 
on $\R^{3n}$ 
that satisfy the estimate 
\begin{equation*}
\left|
\partial^{\alpha}_{x} 
\partial^{\beta}_{\xi} 
\partial^{\gamma}_{\eta} 
\sigma (x, \xi, \eta)
\right| 
\le 
C_{\alpha, \beta, \gamma} 
\left( 1+ |\xi| + |\eta|\right)^{m-|\beta|-|\gamma|}
\end{equation*}
for all multi-indices $\alpha, \beta, \gamma \in (\N_0)^{n}$. 
\end{defn}

For this class, the following theorem holds.

\begin{referthmB}
Let $m \in \R$,
$n\ge 2$,  
$1\le p, q\le \infty$, and $1/p+1/q=1/r$.  
Suppose that 
$\phi_1$ and $\phi_2$ are non-degenerate phase functions
and an amplitude 
$\sigma \in BS^{m}_{1,0}$
is compactly supported in the spatial variable.
Then, if 
\[
m = m_{0}(p, q) := -(n-1) 
\Big( \Big|\frac{1}{p} - \frac{1}{2} \Big|
+\Big|\frac{1}{q} - \frac{1}{2} \Big| \Big) ,
\]
there exists a positive constant $C$ such that
\[
\|T_{\sigma}^{\phi_1, \phi_2}\|_{H^p \times H^q \to L^r}
\le C ,
\]
where $L^r$ should be replaced by $BMO$ when $r=\infty$. 
\end{referthmB}

The boundedness property 
for the bilinear operator $T_{\sigma}^{\phi_1, \phi_2}$
was first investigated by Grafakos--Peloso \cite{GP},
where the above-mentioned boundedness 
for $1 < p, q < 2$ and 
$m < m_{0}(p, q)$ was obtained.
After that, 
Rodr\'iguez-L\'opez--Rule--Staubach \cite{RRS-AM} 
generalized it 
to the full range of $1\le p, q\le \infty$,
and gave Theorem B.
Also, we remark that, in \cite{RS},
Rodr\'iguez-L\'opez--Staubach
considered the bilinear operator 
$T_{\sigma}^{\phi_1, \phi_2}$ 
for an amplitude without compactness of the spatial support
and 
for phase functions with a bit stronger condition,
and gave a global boundedness from $L^p \times L^q$ to $L^r$ 
when $m < m_{0}(p, q)$.
In \cite{RRS-TAMS},
Rodr\'iguez-L\'opez--Rule--Staubach
obtained a global boundedness from $L^2 \times L^2$ to $L^1$ 
for $m = m_{0}(2,2) = 0$.

Quite recently,
for a typical case of phase functions
derived from the wave operators
and for an $x$-independent amplitude 
$\sigma = \sigma(\xi, \eta)$,
the authors \cite{KMT}
improved the order $m$ stated in Theorem B.
To describe the result obtained there, 
we define 
\begin{equation*}
m_1 (p, q)
=
\begin{cases}
{-(n-1) \big( 
|\frac{1}{p}- \frac{1}{2}|+|\frac{1}{q}- \frac{1}{2}| \big) }
& 
\text{if $1\le p, q \le 2$ or if $2\le p, q \le \infty$,} 
\\
{-\big( 
\frac{1}{p} - \frac{1}{2} \big) 
- (n-1) \big( \frac{1}{2} - \frac{1}{q} \big) }
& 
\text{if $1\le p\le 2\le q \le \infty$ and $\frac{1}{p}+\frac{1}{q} \le 1$,}
\\
{-(n-1)
\big( \frac{1}{p} - \frac{1}{2} \big) 
- \big( \frac{1}{2} - \frac{1}{q} \big) }
& 
\text{if $1\le p\le 2\le q \le \infty$ and $\frac{1}{p}+\frac{1}{q}\ge 1$,}
\\
{-(n-1)
\big( 
\frac{1}{2} - \frac{1}{p} \big) 
- \big( \frac{1}{q} - \frac{1}{2} \big) }
& 
\text{if $1\le q\le 2\le p \le \infty$ and $\frac{1}{p}+\frac{1}{q} \le 1$,}
\\
{-\big( \frac{1}{2} - \frac{1}{p} \big) 
- (n-1) \big( \frac{1}{q} - \frac{1}{2} \big) }
& 
\text{if $1\le q\le 2\le p \le \infty$ and $\frac{1}{p}+\frac{1}{q}\ge 1$.}
\end{cases}
\end{equation*}
With this notation,
the following global boundedness holds.

\begin{referthmC}
Let $m \in \R$,
$n\ge 2$,  
$1\le p, q\le \infty$, and $1/p+1/q=1/r$.  
Suppose that
$\phi_{i} (x, \xi) = x \cdot \xi + |\xi|$,
$i=1,2$,
and
$\sigma = \sigma(\xi, \eta) \in BS^{m}_{1,0}$.
Then, 
if $m= m_1 (p, q)$,
there exists a positive constant $C$ such that
\[
\|T_{\sigma}^{\phi_1, \phi_2}\|_{H^p \times H^q \to L^r}
\le C,
\]  
where $L^r$ should be replaced by $BMO$ when $r=\infty$. 
\end{referthmC}

Since $m_0(p, q) \le m_1(p, q)$
for $n \ge 2$ and $1\le p, q \le \infty$,
we see that Theorem C is an improvement of Theorem B.
Also, it should be remarked that 
the number $m_1(p, q)$ is optimal
when $1 \le p, q \le 2$, $2\le p, q \le \infty$,
or $1/p+1/q=1$
(see \cite[Theorem 1.5]{KMT}).

The purpose of this paper is
to improve Theorem B and
to generalize Theorem C to
the boundedness for bilinear Fourier integral operators.
The main theorem reads as follows.

\begin{thm} \label{th-main}
Let $m \in \R$,
$n\ge 2$,  
$1\le p, q\le \infty$, and $1/p+1/q=1/r$.  
Suppose that 
$\phi_1$ and $\phi_2$ are non-degenerate phase functions
and an amplitude 
$\sigma \in BS^{m}_{1,0}$
is compactly supported in the spatial variable.
Then, if $m = m_{1}(p, q) $,
there exists a positive constant $C$ such that
\[
\|T_{\sigma}^{\phi_1, \phi_2}\|_{H^p \times H^q \to L^r}
\le C ,
\]
where $L^r$ should be replaced by $BMO$ when $r=\infty$. 
\end{thm}

The contents of the rest of the paper are as follows. 
In Section \ref{pf-thm}, 
we present and prove Proposition \ref{prop-L1}
concerning the boundedness
from $h^1 \times L^\infty$ to $L^1$,
which is an endpoint case of Theorem \ref{th-main}
and implies Theorem \ref{th-main} 
with the aid of complex interpolation.
In Appendix \ref{app},
we show Lemma \ref{lem-D3} 
to be used for proving Proposition \ref{prop-L1}.

We end this section by introducing 
some notations used throughout this paper.
\begin{notation}\label{notation}
The Fourier transform and the inverse Fourier transform 
on $\R^n$ are defined by 
\begin{align*}&
\widehat{f}(\xi)
=
\int_{\R^n}
e^{-i \xi \cdot x}
f(x)\, dx ,
\\&
(g)^{\vee}(x)
=
\frac{1}{(2\pi)^n} 
\int_{\R^n}
e^{i \xi \cdot x}
g(\xi)\, d\xi. 
\end{align*}

We take $\varphi, \psi \in \calS (\R^n)$ 
satisfying that 
$\varphi = 1$ on 
$\{ |\xi| \le 1 \}$,
$\supp \varphi \subset 
\{ |\xi| \le 2 \}$,
$\supp \psi \subset 
\{ 1/2 \leq |\xi| \leq 2 \}$,
and
$\varphi+\sum_{j\in\N} \psi ( 2^{-j} \cdot ) = 1$.
In what follows, we will write
$\psi_{0} = \varphi$,
$\psi_j = \psi ( 2^{-j} \cdot )$ 
for $j \in \N$, 
and 
$\varphi_j = \varphi ( 2^{-j} \cdot )$ 
for $j \in \N_{0}$.
Then, we see that 
$\varphi_0 = \psi_0 = \varphi$ and
\begin{align*}
\sum_{j=0}^{k} \psi_j = \varphi_k,  
\quad k\in \N_0. 
\end{align*}
We define $C^{\infty}$ function
$\zeta = 1 - \varphi$.
Then we have
$\partial^{\alpha} \zeta 
\in C^{\infty}_{0} (\R^n)$
for $|\alpha|\ge 1$,
$\zeta = \sum_{j \in \N} \psi_{j}$,
and
\begin{align*} &
\zeta =0 \;\; \text{on}\;\; \{ |\xi|\le 1 \},
\quad
\zeta =1 \;\; \text{on}\;\; \{ |\xi|\ge 2 \}.
\end{align*}

For a smooth function $\theta$ on $\R^n$ and 
for $N \in \N_{0}$, we write 
$\|\theta\|_{C^N}=
\max_{|\alpha|\le N} 
\sup_{\xi} \big| \partial_{\xi}^{\alpha} \theta (\xi) \big|$. 

\end{notation}


\section{Proof of Theorem \ref{th-main}}
\label{pf-thm}

By complex interpolation,
Theorem \ref{th-main} follows from
Theorem B and Proposition \ref{prop-L1} below,
and thus, this section is only devoted to proving 
the following statement.

\begin{prop}\label{prop-L1} 
Let $n\ge 2$.
Suppose that 
$\phi_1$ and $\phi_2$ are non-degenerate phase functions
and an amplitude 
$\sigma \in BS^{-n/2}_{1,0}$
is compactly supported in the spatial variable.
Then, 
there exists a positive constant $C$ such that
\[
\|T_{\sigma}^{\phi_1, \phi_2}
\|_{h^1 \times L^\infty \to L^1}
\le C.
\]
\end{prop}

\begin{proof}
Before starting the proof,
we give a remark on 
the non-degeneracy conditions of 
$\phi_1$ and $\phi_2$.
Let $K$ be the closure of 
the set 
$\{ x \in \R^n \mid \sigma (x, \cdot , \cdot) \neq 0 \}$.
Then, 
since $\phi_1$ and $\phi_2$ satisfy
the non-degeneracy conditions 
for $x \in {K}$ and $\xi \in \R^n \setminus\{0\}$,
we see from continuity
that there exists a neighborhood 
$\widetilde{K}$ of the set $K$
such that
\begin{equation} \label{ND-n}
\det \Big( \frac{\partial^2 \phi_i (x,\xi)}{\partial{x} \partial{\xi}} \Big) \neq 0
\;\; \text{for}
\;\; x \in \widetilde{K} ,
\;\; \xi \in \R^n \setminus\{0\} ,
\;\; i = 1,2.
\end{equation}
In the following argument,
$\widetilde{K}$ always denotes
such a neighborhood of $K$.

Now, we shall begin with the proof of the proposition.
We first decompose $\sigma$ as 
\begin{align*} &
\sigma (x, \xi, \eta) 
=
\sum_{j=0}^{\infty} 
\sum_{k=0}^{\infty} 
\sigma (x, \xi, \eta) \psi_j (\xi) \psi_k (\eta)
\\&=
\sigma (x, \xi, \eta) \varphi (\xi) \varphi (\eta)
+
\sum_{j=1}^{\infty} 
\sum_{k=0}^{j} 
\sigma (x, \xi, \eta) \psi_j (\xi) \psi_k (\eta)
+
\sum_{k=1}^{\infty} 
\sum_{j=0}^{k-1} 
\sigma (x, \xi, \eta) \psi_j (\xi) \psi_k (\eta)
\\&=
\sigma (x, \xi, \eta) \varphi (\xi) \varphi (\eta)
+
\sum_{j \in \N} 
\sigma (x, \xi, \eta) \psi_{j} (\xi) \varphi_{j} (\eta)
+
\sum_{k \in \N} 
\sigma (x, \xi, \eta) \varphi_{k-1} (\xi) \psi_k (\eta)
\\&=
\sigma_{0} (x, \xi, \eta)
+
\sigma_{\RomI} (x, \xi, \eta)
+
\sigma_{\II} (x, \xi, \eta) .
\end{align*}

We consider the amplitude $\sigma_{\RomI}$. 
Take
a function $\chi \in C_{0}^{\infty}(\R^n)$ such that
\[
\chi = 1 \;\;\text{on}\;\; K,
\quad
\supp \chi 
\subset \widetilde{K} ,
\]
and functions 
$\widetilde{\psi}, \widetilde{\varphi} \in C_{0}^{\infty}(\R^n)$ 
such that 
\[\begin{array}{ll}
\widetilde{\psi} = 1 
\;\;\text{on}\;\; \{ 2^{-1}\le |\xi| \le 2 \}, \quad
&
\supp \widetilde{\psi} 
\subset \{3^{-1}\le |\xi| \le 3\}, 
\vspace{3pt}\\
\widetilde{\varphi} = 1 
\;\;\text{on}\;\; \{ |\xi| \le 2 \}, \quad
&
\supp \widetilde{\varphi} 
\subset \{|\xi| \le 3\} .
\end{array}\]
Then we can write
\[
\sigma_{\RomI}(x, \xi, \eta)
=
\sum_{j \in \N }
\sigma (x, \xi, \eta) 
\widetilde{\psi}(2^{-j}\xi) 
\widetilde{\varphi}(2^{-j}\eta) \,
\chi(x) \psi_{j} (\xi) 
\chi(x) \varphi_{j} (\eta). 
\]
Since $\sigma \in BS_{1,0}^{-n/2}$ and
\[
\supp \sigma (x, 2^{j} \xi, 2^{j}\eta)\widetilde{\psi}(\xi) 
\widetilde{\varphi}(\eta)
\subset 
K \times 
\{3^{-1}\le |\xi|\le 3\}\times 
\{|\eta|\le 3\},
\]
the following estimate holds:
\begin{equation*}
\left|
\partial^{\alpha}_{\xi} 
\partial^{\beta}_{\eta} 
\big(
\sigma (x, 2^{j} \xi, 2^{j}\eta)\widetilde{\psi}(\xi) 
\widetilde{\varphi}(\eta)
\big) 
\right| 
\le 
C_{\alpha, \beta}\,  2^{-j n/2}
\end{equation*}
with $C_{\alpha, \beta}$ independent of $j\in \N_0$. 
Hence, by the Fourier series expansion
with respect to the variables $\xi$ and $\eta$, 
we can write 
\[
\sigma (x, 2^{j} \xi, 2^{j}\eta)\widetilde{\psi}(\xi) 
\widetilde{\varphi}(\eta)
=
\sum_{a, b \in \Z^n} 
c_{\RomI, j}^{(a, b)} (x)
e^{i a \cdot \xi} e^{i b \cdot \eta}, 
\quad 
|\xi|<\pi, \;\; |\eta|<\pi,  
\]
with the coefficient satisfying that
for $L>0$
\begin{equation} \label{cjab-decay}
| c_{\RomI, j}^{(a, b)} (x) | 
\lesssim 2^{-j n/2} 
(1+|a|)^{-L} (1+|b|)^{-L},
\quad
x \in \R^n.
\end{equation}
Changing variables 
$\xi \to 2^{-j}\xi$ and $\eta \to 2^{-j}\eta$
and 
multiplying $(\chi(x))^2 \psi_{j} (\xi) \varphi_{j} (\eta)$, 
we obtain 
\[
\sigma (x, \xi, \eta)
\psi_{j} (\xi) 
\varphi_{j} (\eta)
=
\sum_{a, b \in \Z^n} 
c_{\RomI, j}^{(a, b)} (x) \,
e^{i a \cdot 2^{-j} \xi} e^{i b \cdot 2^{-j}\eta}
\chi(x) \psi_{j} (\xi) \chi(x) \varphi_{j} (\eta).  
\]
Hence, by the definitions of
$\psi_{j}$ and $\varphi_{j}$
in Notation \ref{notation}, 
the amplitude $\sigma_{\RomI}$ is written as
\begin{align*}
\sigma_{\RomI} (x, \xi, \eta)
&=
\sum_{a, b \in \Z^n}
\sum_{j \in \N }
c_{\RomI, j}^{(a, b)} (x) \,
e^{i a \cdot 2^{-j} \xi} e^{i b \cdot 2^{-j}\eta}
\chi(x) \psi (2^{-j} \xi) \,
\chi(x) \varphi (2^{-j} \eta) 
\\
&
=
\sum_{a, b \in \Z^n}
\sigma_{\RomI}^{(a, b)} (x, \xi, \eta), 
\end{align*}
where 
\[
\sigma_{\RomI}^{(a, b)} (x, \xi, \eta) =
\sum_{j \in \N }
c_{\RomI, j}^{(a, b)} (x) \,
\Psi^{(a)} (x, 2^{-j} \xi) \,
\Phi^{(b)} (x, 2^{-j} \eta)
\]
and 
\begin{equation*}
\Psi^{(\nu)} (x, \xi)
= \chi(x) \, e^{i \nu \cdot \xi} \psi (\xi), 
\quad  
\Phi^{(\nu)} (x, \eta) 
= \chi(x) \, e^{i \nu \cdot \eta} \varphi (\eta),
\quad
\nu \in \Z^n .
\end{equation*}

By similar arguments, 
$\sigma_{0}$ can be written as
\begin{align*}
&
\sigma_{0} (x, \xi, \eta)
=
\sum_{a, b \in \Z^n}
c_{0}^{(a, b)} (x) \,
\Phi^{(a)} (x, \xi) \,
\Phi^{(b)} (x, \eta) ,
\end{align*}
where the coefficient
$c_{0}^{(a, b)}$ 
satisfies the same condition as in \eqref{cjab-decay} with $j=0$,
and
$\sigma_{\II}$ can be written as
\begin{align*}
\sigma_{\II} (x, \xi, \eta)
=
\sum_{a, b \in \Z^n}
\sum_{j \in \N }
c_{\II, j}^{(a, b)} (x) \,
\Phi^{(a)} (x, 2^{-(j-1)} \xi) \,
\Psi^{(b)} (x, 2^{-j} \eta) ,
\end{align*}
where the coefficient 
$c_{\II, j}^{(a, b)}$ 
satisfies the same condition as in \eqref{cjab-decay}.

Hereafter we shall consider 
the amplitudes $\widetilde{\sigma}_0$ and $\widetilde{\sigma}$
of the forms
\begin{align}&\label{assumption-sigma_0}
\widetilde{\sigma}_0 (x, \xi, \eta) 
=
c_{0} (x) \,
\Theta_{1} (x, \xi) \,
\Theta_{2} (x, \eta) ,
\\
&\label{assumption-sigma}
\widetilde{\sigma} (x, \xi, \eta) 
=
\sum_{j \in \N}
c_{j} (x) \,
\Theta_{1} (x, 2^{-j} \xi) \,
\Theta_{2} (x, 2^{-j} \eta), 
\end{align}
where 
$( c_j (x) )_{j\in \N_0}$ is 
a sequence of complex-valued functions satisfying 
\begin{equation}\label{assumption-cj}
\left| c_{j} (x) \right| 
\le 2^{-j n/2} A, 
\quad 
x \in \R^n,
\;\;
j\in \N_0, 
\end{equation}
with some $A\in (0, \infty)$,  
and $\Theta_{1}$ and $\Theta_{2}$ are 
smooth functions on $\R^{2n}$ given by
\begin{equation}\label{assumption-Theta}
\Theta_{1} (x, \xi) = \chi (x) \theta_1 (\xi), \quad
\Theta_{2} (x, \eta) = \chi (x) \theta_2 (\eta)
\end{equation}
with $\chi, \theta_1, \theta_2 \in \calS(\R^n)$ 
satisfying that 
\begin{equation}\label{assumption-theta}
\supp \chi \subset \widetilde{K},
\quad
\supp \theta_1 , \, \supp \theta_2 
\subset 
\{|\xi|\le 2 \} 
\end{equation}
(for the set $\widetilde{K}$,
see the argument around \eqref{ND-n}).
For such $\widetilde{\sigma}_0$ and $\widetilde{\sigma}$, 
we shall prove that
there exist constants
$c \in (0, \infty)$ and $N \in \N$
such that
\begin{equation}\label{Goal-estimate}
\| T_{ \widetilde{\sigma}_0 }^{\phi_1, \phi_2}
\|_{h^1 \times L^{\infty}\to L^1}, 
\;
\| T_{ \widetilde{\sigma} }^{\phi_1, \phi_2}
\|_{h^1 \times L^{\infty}\to L^1}
\le 
c A \|\theta_1\|_{C^N} \|\theta_2\|_{C^N} .
\end{equation}

If this is proved, by applying
\eqref{Goal-estimate} for $\widetilde{\sigma}$ to 
$c_j = c_{\RomI, j}^{(a, b)}$,  
$\Theta_{1} = \Psi^{(a)}$, and 
$\Theta_{2} = \Phi^{(b)}$,
we have 
\begin{align*}
\| T_{ {\sigma}_{\RomI}^{(a, b)} }^{\phi_1, \phi_2}
\|_{ h^1 \times L^{\infty}\to L^1 }
&\lesssim 
(1+|a|)^{-L} (1+|b|)^{-L} 
\|e^{i a \cdot \xi} \psi (\xi)\|_{C^{N}_{\xi}} 
\|e^{i b \cdot \eta} \varphi (\eta)\|_{C^{N}_{\eta}} 
\\&\lesssim (1+|a|)^{-L+N}
(1+|b|)^{-L+N}, 
\end{align*}
and, thus, taking $L$ sufficiently large and taking sum over 
$a, b \in \Z^n$, 
we obtain 
\[
\left\| 
T_{ {\sigma}_{\RomI} }^{\phi_1, \phi_2}
\right\|_{ h^1 \times L^{\infty}\to L^1 } 
\lesssim 1.
\]
In the same way, 
we obtain 
\[
\| T_{ {\sigma}_{0} }^{\phi_1, \phi_2}
\|_{ h^1 \times L^{\infty}\to L^1 } ,
\;
\| T_{ {\sigma}_{\II} }^{\phi_1, \phi_2}
\|_{ h^1 \times L^{\infty}\to L^1 } 
\lesssim 1.
\]
The above three estimates complete 
the proof of the proposition.

Thus the proof is reduced to showing \eqref{Goal-estimate} 
for $\widetilde{\sigma}_0$ and $\widetilde{\sigma}$ 
given by \eqref{assumption-sigma_0}--\eqref{assumption-theta}.
Using the smooth functions 
$\varphi$ and $\zeta$ 
of Notation \ref{notation}, 
we decompose the amplitude
$\widetilde{\sigma}$ into two parts:  
\begin{align*}
&
\widetilde{\sigma} (x, \xi, \eta)
=
\tau_1 (x, \xi, \eta) + \tau_2 (x, \xi, \eta) ,
\\
&
\tau_1 (x, \xi, \eta) = 
\varphi (\eta) \widetilde{\sigma} (x, \xi, \eta), 
\quad
\tau_2 (x, \xi, \eta) = 
\zeta (\eta) 
\widetilde{\sigma} (x, \xi, \eta) .  
\end{align*}
Then,
bilinear Fourier integral operators with 
the amplitudes
$\widetilde{\sigma}_0, \tau_1, \tau_2$
are expressed,
by using linear ones,
as
\begin{align}
\label{T_0}
T_{ \widetilde{\sigma}_0 }^{\phi_1, \phi_2} (f, g)(x)
&=
c_{0} (x) \,
T^{\phi_1}_{ \Theta_{1} } f(x) \;
T^{\phi_2}_{ \Theta_{2} } g(x), 
\\ 
\label{T_tau1}
T_{ \tau_1 }^{\phi_1, \phi_2} (f, g)(x)
&= \sum_{j \in \N } 
c_{j} (x) \,
T^{\phi_1}_{ \Theta_{1} (x, 2^{-j} \xi) } f(x) \;
T^{\phi_2}_{ \Theta_{2} (x, 2^{-j} \eta) \varphi(\eta) } g(x), 
\\ 
\label{T_tau2}
T_{ \tau_2 }^{\phi_1, \phi_2} (f, g)(x)
&= \sum_{j \in \N } 
c_{j} (x) \,
T^{\phi_1}_{ \Theta_{1} (x, 2^{-j} \xi) } f(x) \;
T^{\phi_2}_{ \Theta_{2} (x, 2^{-j} \eta) \zeta(\eta) } g(x).
\end{align}

We shall consider the amplitudes and the phase functions
of the linear Fourier integral operators
appearing above.
For the amplitudes,
we observe that,
if $m \le 0$ and $N \in \N_{0}$,
then
\begin{align*}&
\| \Theta_{i}(x, 2^{-j} \xi) \|_{ S_{1,0}^{m} (N) }
=
\| \chi(x) \, \theta_i (2^{-j} \xi ) \|_{ S_{1,0}^{m} (N) }
\lesssim 
2^{-j m} \| \theta_i \|_{C^N},
\quad
i=1,2,
\\&
\| \Theta_{2} (x, 2^{-j} \eta) \varphi(\eta) \|_{ S_{1,0}^{m} (N) }
=
\| \chi(x) \, \theta_2 (2^{-j} \eta ) \varphi (\eta) \|_{ S_{1,0}^{m} (N) }
\lesssim 
\| \theta_2 \|_{C^N},
\\&
\| \Theta_{2} (x, 2^{-j} \eta) \zeta(\eta) \|_{ S_{1,0}^{m} (N) }
=
\| \chi(x) \, \theta_2 (2^{-j} \eta ) \zeta(\eta) \|_{ S_{1,0}^{m} (N) }
\lesssim 
2^{-j m} \| \theta_2 \|_{C^N} 
\end{align*}
hold for all $j \in \N_0$.
In fact, 
since $|\xi| \le 2^{j+1}$ holds
on the support of $\theta_{i} (2^{-j} \cdot)$,
if $m \le 0$ and $\beta \in (\N_{0})^n$,
then we have
\[ 
( 1+ |\xi| )^{-(m-|\beta|)} \lesssim 2^{-j(m-|\beta|)} ,
\quad j \in \N_{0},
\]
which implies the first inequality.
By using the facts that
$\varphi$ and $\partial^{\alpha} \zeta$,
$|\alpha| \ge 1$, belong to $C^{\infty}_{0} (\R^n)$,
the remaining two estimates
can be obtained.
For the phase functions,
we see 
from \eqref{ND-n}, \eqref{assumption-Theta},
and \eqref{assumption-theta}
that
$\phi_{i} (x, \xi)$, $i = 1,2$, 
satisfies the non-degeneracy condition
for $x$ in the spatial support of 
$\Theta_{i} (x, 2^{-j} \xi)$ 
and 
for $\xi \in \R^n \setminus \{0\}$.
Moreover, from the assumption,
they are real-valued and 
positively homogeneous of degree one 
in the frequency variable.

Hence, 
we can apply
Theorem A to 
the linear Fourier integral operators
in \eqref{T_0}, \eqref{T_tau1}, and \eqref{T_tau2}:
if 
\[
m_{\epsilon} 
= - \frac{n-1}{2} - \epsilon,
\quad
0 < \epsilon < 1/2,
\]
then 
there exists $N \in \N$
such that
\begin{align}
&\label{f-L^1-bdd}
\| T^{\phi_1}_{ \Theta_{1} (x, 2^{-j} \xi) } \|_{ L^{1} \to L^{1} } 
\lesssim 
2^{ -j m_\epsilon } 
\| \theta_1 \|_{ C^N } ,
\\
&\label{f-L^2-bdd}
\| T^{\phi_1}_{ \Theta_{1} (x, 2^{-j} \xi) } \|_{ L^{2} \to L^{2} } 
\lesssim 
\| \theta_1 \|_{ C^N } ,
\\
&\label{g-L^infty-bdd}
\| T^{\phi_2}_{ \Theta_{2} (x, 2^{-j} \eta)} \|_{ L^{\infty} \to L^{\infty} }
\lesssim 
2^{ -j m_\epsilon } 
\| \theta_2 \|_{ C^N } ,
\\
&\label{g1-L^infty-bdd}
\| T^{\phi_2}_{ \Theta_{2} (x, 2^{-j} \eta) \varphi(\eta) } \|_{ L^{\infty} \to L^{\infty} }
\lesssim 
\| \theta_2 \|_{ C^N } ,
\\
&\label{g2-L^2-bdd}
\| T^{\phi_2}_{ \Theta_{2} (x, 2^{-j} \eta) \zeta(\eta) } \|_{ L^{2} \to L^{2} }
\lesssim 
\| \theta_2 \|_{ C^N } ,
\\
&\label{g2-L^infty-bdd}
\| T^{\phi_2}_{ \Theta_{2} (x, 2^{-j} \eta) \zeta(\eta) } \|_{ L^{\infty} \to L^{\infty} }
\lesssim 
2^{- j m_\epsilon} \| \theta_2 \|_{ C^N } 
\end{align}
hold for all $j \in \N_0$.

Now, we shall actually estimate the bilinear operators
in \eqref{T_0}, \eqref{T_tau1}, and \eqref{T_tau2}.
Estimates for $\widetilde{\sigma}_0$ and $\tau_1$ are simple.
In fact, we have by 
\eqref{T_0},
\eqref{f-L^1-bdd}, and
\eqref{g-L^infty-bdd} for $j=0$
\begin{align*}
\| T_{ \widetilde{\sigma}_0 }^{\phi_1, \phi_2} (f, g) \|_{L^1}
&\le A
\| T^{\phi_1}_{ \Theta_{1} } f \|_{L^1} \,
\| T^{\phi_2}_{ \Theta_{2} } g \|_{L^{\infty}} 
\\&\lesssim A
\| \theta_1 \|_{ C^N } 
\| \theta_2 \|_{ C^N } 
\|f\|_{L^1} \|g\|_{L^\infty},
\end{align*}
and by 
\eqref{T_tau1},
\eqref{f-L^1-bdd}, and
\eqref{g1-L^infty-bdd}
\begin{align*}
\| T_{ \tau_1 }^{\phi_1, \phi_2} (f, g) \|_{L^1}
&\le A
\sum_{j \in \N } 
2^{-j n/2} \,
\| T^{\phi_1}_{ \Theta_1 (x, 2^{-j} \xi) } f \|_{L^1} \,
\| T^{\phi_2}_{ \Theta_2 (x, 2^{-j} \eta) \varphi(\eta) } g \|_{L^{\infty}} 
\\&\lesssim A
\sum_{j \in \N } 
2^{-j n/2} \, 2^{j(\frac{n-1}{2} + \epsilon)} 
\| \theta_1 \|_{ C^N } \|f\|_{L^1}
\| \theta_2 \|_{ C^N } \|g\|_{L^\infty}
\\&\lesssim A
\| \theta_1 \|_{ C^N } \| \theta_2 \|_{ C^N } 
\|f\|_{L^1} \|g\|_{L^\infty} .
\end{align*}
These imply the desired estimates for 
$\widetilde{\sigma}_0$ and $\tau_1$ 
with the embedding $h^1 \hookrightarrow L^1$.

Next, we consider the amplitude
$\tau_2$.
To obtain the desired inequality, 
it suffices to prove
\[
\left\|
T_{\tau_2}^{\phi_1, \phi_2} (f, g) 
\right\|_{L^1}
\lesssim 
A
\|\theta_1\|_{C^{N}}
\|\theta_2\|_{C^{N}}
 \|g\|_{L^{\infty}}
\]
for $h^1$-atoms $f$, which satisfy that
\[
\supp f \subset \{ y \in \R^n \mid |y - \bar{y}| \le r \} , 
\quad 
\|f\|_{L^{\infty}}\le r^{-n}, 
\]
and, in addition, if $r \le 1$
\[
\int f(y)\, dy = 0.
\]

For the case $r \ge 1$,
since $\|f\|_{L^2} \lesssim r^{-n/2} \le 1$,
we have by
\eqref{T_tau2},
\eqref{f-L^2-bdd}, and
\eqref{g2-L^infty-bdd}
\begin{align*}&
\| T_{ \tau_2 }^{\phi_1, \phi_2} (f, g) \|_{L^1}
\le
| \widetilde{K} |^{1/2}
\| T_{ \tau_2 }^{\phi_1, \phi_2} (f, g) \|_{L^2}
\\&\lesssim A
\sum_{j \in \N } 
2^{-j n/2} \,
\| T^{\phi_1}_{ \Theta_1 (x, 2^{-j} \xi) } f \|_{L^2} \,
\| T^{\phi_2}_{ \Theta_2 (x, 2^{-j} \eta) \zeta(\eta) } g \|_{L^{\infty}} 
\\&\lesssim A
\sum_{j \in \N } 
2^{-j n/2} \, 2^{j(\frac{n-1}{2} + \epsilon)} 
\| \theta_1 \|_{ C^N } \|f\|_{L^2}
\| \theta_2 \|_{ C^N } \|g\|_{L^\infty}
\\&\lesssim A
\| \theta_1 \|_{ C^N } \| \theta_2 \|_{ C^N } 
\|g\|_{L^\infty} .
\end{align*}

For the case $r \le 1$, we will use the following lemma.
This lemma is essentially proved in \cite[Lemma 4.4]{KMT},
but we will give its proof in Appendix A
for the reader's convenience.
\begin{lem}\label{lem-D3} 
Let $n\ge 2$ and $\calK \subset \R^n$ be a compact set, 
and let $\zeta$ be the smooth function 
given in Notation \ref{notation} 
and let 
$\theta \in C_{0}^{\infty} (\R^{n})$ 
satisfy $\supp \theta \subset \{|\xi|\le 2\}$.
Suppose that 
$\phi$ is a non-degenerate phase function 
and set 
$R = 1+ \sup \big\{ |\nabla_{\xi} \phi (x, \xi)| \mid x \in \calK, |\xi|=1 \big\}$.  
Then for each positive integer $N>n$, 
there  exists a positive constant $C$
such that 
\[
\left| \int_{\R^n} e^{-i y \cdot \xi + i \phi (x, \xi)} 
\zeta(\xi) \theta (2^{-j}\xi) \,d\xi \right| 
\le C \,
\|\theta\|_{C^N} |y|^{-N} 
\]
for $x \in \calK$, $|y| \ge 2R$, and $j\in \N$. 
\end{lem}

Now, we shall prove the case $r \le 1$.
We set 
\[
R = 1+ \sup \{ |\nabla_{\xi} \phi_2 (x, \xi)| 
\mid x \in \widetilde{K}, \; |\xi|=1 \}
\]
and decompose 
$g$ as $g=g\ichi_{B}+g\ichi_{B^c}$,
where
$B=\{x \in \R^n \mid |x|\le 2R\}$. 
We first consider the estimate
for $T_{ \tau_2 }^{\phi_1, \phi_2} (f, g\ichi_{B^c})$.
By the kernel representation of 
(linear) Fourier integral operators,
we have
\begin{align*}&
T^{\phi_2}_{ \Theta_2 (x, 2^{-j} \eta) \zeta(\eta) } (g\ichi_{B^c})(x)
\\&= 
\int_{|y| \ge 2R} g(y) 
\Big\{ \chi(x) 
\int_{\R^n} e^{-i y \cdot \eta + i \phi_2 (x, \eta)} 
\zeta(\eta) \theta_2 (2^{-j}\eta) \,d\eta \Big\}
dy .
\end{align*}
Then, Lemma \ref{lem-D3} yields that
\begin{align*}
\sup_{x \in \widetilde{K}}
\big| T^{\phi_2}_{ \Theta_2 (x, 2^{-j} \eta) \zeta(\eta) } (g\ichi_{B^c})(x) \big|
&\lesssim
\int_{|y| \ge 2R} | g(y) | 
\|\theta_2\|_{C^N} |y|^{-N} 
\,dy
\\&\lesssim 
\|\theta_2\|_{C^N} \|g\|_{L^\infty} .
\end{align*}
Therefore, since $\|f\|_{L^1} \lesssim 1$, we have by 
\eqref{T_tau2}, 
\eqref{f-L^1-bdd}, and 
the above estimate
\begin{align*}
\| T_{ \tau_2 }^{\phi_1, \phi_2} (f, g\ichi_{B^c}) \|_{L^1}
&\le A
\sum_{j \in \N } 
2^{-j n/2} \,
\| T^{\phi_1}_{ \Theta_1 (x, 2^{-j} \xi) } f \|_{L^1} \,
\| T^{\phi_2}_{ \Theta_2 (x, 2^{-j} \eta) \zeta(\eta) } (g\ichi_{B^c}) \|_{L^{\infty}(\widetilde{K})} 
\\&\lesssim A
\sum_{j \in \N } 
2^{-j n/2} \, 2^{j(\frac{n-1}{2} + \epsilon)} 
\| \theta_1 \|_{ C^N } \|f\|_{L^1}
\| \theta_2 \|_{ C^N } \|g\|_{L^\infty}
\\&\lesssim A
\| \theta_1 \|_{ C^N } \| \theta_2 \|_{ C^N } 
\|g\|_{L^\infty} .
\end{align*}

We finally consider the estimate
for $T_{ \tau_2 }^{\phi_1, \phi_2} (f, g\ichi_{B})$.
We notice from \eqref{f-L^2-bdd} that
\begin{align} \label{tau2-f-L^2}
\| T^{\phi_1}_{ \Theta_1 (x, 2^{-j} \xi) } f \|_{L^2} 
\lesssim 
\| \theta_1 \|_{ C^N } \|f\|_{ L^2 }
\lesssim 
\| \theta_1 \|_{ C^N } 2^{j n/2} \big( 2^{j} r \big)^{-n/2} .
\end{align}
Moreover, we have by the kernel representation
\[
T^{\phi_1}_{ \Theta_1 (x, 2^{-j} \xi) } f(x)
= 
\int_{\R^n} f(y) K_j(x, y) \,dy,
\]
where
\[
K_j(x, y)
= 
\int_{\R^n} e^{-i y \cdot \xi + i \phi_1 (x, \xi)} 
\chi(x) \theta_1 (2^{-j}\xi) \,d\xi .
\]
Here, by the moment condition of $h^1$-atoms,
we have
\begin{align*}&
T^{\phi_1}_{ \Theta_1 (x, 2^{-j} \xi) } f(x)
= 
\int_{\R^n} f(y) \Big\{ K_j(x, y) - K_j(x,\bar{y}) \Big\}
\,dy
\\& = 
\sum_{|\alpha|=1}
\iint_{ \substack{ |y-\bar{y}|\le r, \\ 0<t<1} }
f(y) \, (y-\bar{y})^{\alpha} \,
\big( \partial_{y}^{\alpha} K_j \big) \big(x, \bar{y}+t(y-\bar{y}) \big) 
\,dt dy ,
\end{align*}
where $\bar{y} \in \R^n$ is the center of the ball containing $\supp f$.
By taking a function $\widetilde{\theta} \in C^{\infty}_{0} (\R^n)$
such that $\widetilde{\theta} = 1$ on $\supp \theta_1$,
the kernel $\partial_{y}^{\alpha} K_j $ can be regarded as
\begin{align*}
\big( \partial_{y}^{\alpha} K_j \big) (x, y)
&=(-i)^{|\alpha|}
\int_{\R^n} e^{ i \phi_1 (x, \xi)} 
\chi(x) \theta_1 (2^{-j}\xi) 
\Big\{ e^{-i y \cdot \xi } \xi^{\alpha} 
\widetilde{\theta} (2^{-j}\xi) \Big\} 
\,d\xi 
\\&=(-i)^{|\alpha|} \,
T^{\phi_1}_{ \Theta_1 (x, 2^{-j} \xi) } 
\big( [e^{-i y \cdot \xi } \xi^{\alpha} 
\widetilde{\theta} (2^{-j}\xi)]^{\vee} \big)(x),
\end{align*}
which implies from 
\eqref{f-L^2-bdd} and Plancherel's theorem
that
\begin{align*}
\| \big( \partial_{y}^{\alpha} K_j \big) (x, y) \|_{L^2_{x}} 
&=
\| T^{\phi_1}_{ \Theta_1 (x, 2^{-j} \xi) } 
\big( [e^{-i y \cdot \xi } \xi^{\alpha} 
\widetilde{\theta} (2^{-j}\xi)]^{\vee} \big)(x) 
\|_{L^2_{x}} 
\\&\lesssim 
\| \theta_1 \|_{ C^N } 
\|[e^{-i y \cdot \xi } \xi^{\alpha} 
\widetilde{\theta} (2^{-j}\xi)]^{\vee} 
\|_{ L^2 }
\\&\approx
2^{j n/2} \, 2^{j|\alpha|}
\| \theta_1 \|_{ C^N } 
\end{align*}
for any $y\in\R^n$.
Hence, we have
\begin{align} \label{tau2-fmom-L^2}
\begin{split}&
\| T^{\phi_1}_{ \Theta_1 (x, 2^{-j} \xi) } f \|_{L^2} 
\\&\le
\sum_{|\alpha|=1}
\iint_{ \substack{ |y-\bar{y}|\le r, \\ 0<t<1} }
|f(y)| \, |y-\bar{y}| \,
\Big\| \big( \partial_{y}^{\alpha} K_j \big) 
\big(x, \bar{y}+t(y-\bar{y}) \big) \Big\|_{L^2_{x}}
\,dt dy
\\&\lesssim
2^{j n/2} \, (2^{j} r)
\| \theta_1 \|_{ C^N } .
\end{split}
\end{align}
Combining \eqref{tau2-f-L^2} and \eqref{tau2-fmom-L^2},
we obtain
\begin{align} \label{f-L^2-bdd-new}
\| T^{\phi_1}_{ \Theta_1 (x, 2^{-j} \xi) } f \|_{L^2} 
\lesssim
2^{j n/2} 
\min\{ (2^{j} r)^{-n/2}, (2^{j} r) \}
\| \theta_1 \|_{ C^N } .
\end{align}
Therefore,
since $\| g\ichi_{B} \|_{L^{2}} \lesssim \|g\|_{L^{\infty}}$, 
we have by 
\eqref{T_tau2}, 
\eqref{f-L^2-bdd-new}, and 
\eqref{g2-L^2-bdd}
\begin{align*}&
\| T_{ \tau_2 }^{\phi_1, \phi_2} (f, g\ichi_{B}) \|_{L^1}
\\&\le A
\sum_{j \in \N } 
2^{-j n/2} \,
\| T^{\phi_1}_{ \Theta_1 (x, 2^{-j} \xi) } f \|_{L^2} \,
\| T^{\phi_2}_{ \Theta_2 (x, 2^{-j} \eta) \zeta(\eta) } (g\ichi_{B}) \|_{L^{2}} 
\\&\lesssim A
\sum_{j \in \N } 
2^{-j n/2} \, 2^{j n/2} 
\min\{ (2^{j} r)^{-n/2}, (2^{j} r) \}
\| \theta_1 \|_{ C^N }
\| \theta_2 \|_{ C^N } \|g\ichi_{B}\|_{L^2}
\\&\lesssim A
\| \theta_1 \|_{ C^N } \| \theta_2 \|_{ C^N } 
\|g\|_{L^\infty} .
\end{align*}
This completes the proof of Proposition \ref{prop-L1}.
\end{proof}

\appendix\section{}\label{sectionappendix}
\label{app}

In this appendix, we prove Lemma \ref{lem-D3}.
To this end, we will use the lemma below, 
which is essentially proved in \cite[Lemma 4.3]{KMT}.

\begin{lem}\label{lem-220925}
Let $n\ge 2$ and $\calK \subset \R^n$ be a compact set, 
and let $\psi$ be a $C^{\infty}$ function on $\R^n$ 
satisfying $\supp \psi \subset \{2^{-1}\le |\xi|\le 2\}$.
Suppose that 
$\phi$ is a non-degenerate phase function 
and set 
$R = 1+ \sup \big\{ |\nabla_{\xi} \phi (x, \xi)| \mid x \in \calK, |\xi|=1 \big\}$.  
Then for each positive integer $N$, 
there  exists a positive constant $C$ 
such that 
\[
\left| \int_{\R^n} e^{-i y \cdot \xi + i \phi (x, \xi)} 
\psi (2^{-j}\xi) \,d\xi \right| 
\le C \,
\|\psi\|_{C^N} (2^j)^{n-N} |y|^{-N} 
\]
for $x \in \calK$, $|y| \ge 2R$, and $j\in \N$. 
\end{lem}

\begin{proof} 
In this proof, 
$I_{j} (x, y)$ denotes
the integral of the left hand side.
By changing variables,
\[
I_{j} (x, y) 
= 2^{j n}
\int_{\R^n} e^{-i 2^{j} ( y \cdot \xi - \phi (x, \xi) )} 
\psi (\xi) \,d\xi .
\]
For $j \in \N$ and $\ell = 1, \dots, n$,
we write
\[
F_{j} 
= 2^{j} \big( y \cdot \xi - \phi (x, \xi) \big),
\quad
G_{j, \ell} 
= \frac{ \partial_{\xi_{ \ell} } F_{j} }{| \nabla_{\xi} F_{j} |^2 } .
\]
Then, since
\[
e^{-i F_{j} } 
= i \sum_{\ell = 1}^{n} 
\frac{ \partial_{\xi_{ \ell} } F_{j} }{| \nabla_{\xi} F_{j} |^2 } \, 
\partial_{\xi_{ \ell} } e^{-i F_{j} } 
= i \sum_{\ell = 1}^{n} 
G_{j, \ell} \, 
\partial_{\xi_{ \ell} } e^{-i F_{j} } ,
\]
we have by integration by parts
\begin{align*}
I_{j} (x, y) 
&= 2^{j n}
\int e^{-i F_{j}} 
\psi \,d\xi 
= 2^{j n} 
\int 
\Big( i \sum_{\ell = 1}^{n} 
G_{j, \ell} \, 
\partial_{\xi_{ \ell} } e^{-i F_{j} } \Big)
\psi \,d\xi
\\&= 2^{j n} \sum_{\ell = 1}^{n} (- i)
\int e^{-i F_{j} } \,
\partial_{\xi_{ \ell} } \big( G_{j, \ell} \, \psi \big)\,d\xi .
\end{align*}
Repeating integration by parts $N$-times,
we have
\begin{align*}
I_{j} (x, y) 
= 
2^{j n} 
\sum_{ \substack { \ell_{1},\dots, \ell_{N} \in \{1, \dots, n \}, \\
|\alpha_{1} + \dots + \alpha_{N} +\beta|=N } }
c
\int e^{-i F_{j} }
\big( \partial_{\xi}^{\alpha_{1}} G_{j, \ell_1} \big) \dots
\big( \partial_{\xi}^{\alpha_{N}} G_{j, \ell_{N}} \big)
\big( \partial_{\xi}^{\beta} \psi \big)
\,d\xi .
\end{align*}
Now, we observe that,
for $|\alpha| \ge 2$,
\begin{align*}&
|\nabla_{\xi} F_{j}| 
= 2^{j} |y-\nabla_{\xi} \phi(x, \xi)| 
\approx 2^{j} |y|,
\\&
| \partial_{\xi}^{\alpha} F_{j} |
= 2^{j} |\partial_{\xi}^{\alpha} \phi (x, \xi) |
\lesssim 2^{j} 
\le 2^{j} |y| ,
\end{align*}
for $x \in \calK$, $|y| \ge 2R$, 
and $1/2 \le |\xi| \le 2$.
Then, we see that
for any
$\alpha \in (\N_{0})^{n}$
\[
| \partial_{\xi}^{\alpha} G_{j, \ell} | 
\lesssim ( 2^{j} |y| )^{-1}
\]
holds.
Therefore, we obtain
\begin{align*}
| I_{j} (x, y) |
&\lesssim 2^{j n} 
\int_{1/2 \le |\xi| \le 2}
( 2^{j} |y| )^{-N}
\| \psi \|_{C^{N}} \,d\xi 
\\&\approx
\|\psi\|_{C^N} (2^j)^{n-N} |y|^{-N} ,
\end{align*}
which completes the proof.
\end{proof}

\begin{proof}[Proof of Lemma \ref{lem-D3}]
From the property of $\zeta$ stated in Notation \ref{notation}, 
we have 
\begin{align*}
\zeta (\xi) \theta (2^{-j} \xi) 
=
\sum_{k=1}^{j+1} 
\psi (2^{-k} \xi) \theta (2^{-j}\xi) .
\end{align*}
If $|y| \ge 2R$ and $1 \le k\le j+1$, then 
Lemma \ref{lem-220925} gives 
\begin{align*}
\left| \int_{\R^n} e^{-i y \cdot \xi + i \phi (x, \xi)} 
\psi (2^{-k}\xi) \theta (2^{-j}\xi) \,d\xi \right| 
&
\lesssim 
(2^{k}) ^{n-N}
|y|^{-N}
\left\| 
\psi (\cdot ) \theta (2^{k-j}\cdot )
\right\|_{C^{N}} 
\\
&
\lesssim 
(2^{k}) ^{n-N}
|y|^{-N}
\left\| 
\theta 
\right\|_{C^{N}} . 
\end{align*}
Hence, if $N>n$, then taking sum over $k$, 
we complete the proof. 
\end{proof}



\begin{thebibliography}{99}

\bibitem[BBMNT]{BBMNT}
\'A B\'enyi,  F. Bernicot, D. Maldonado, V. Naibo, and R.H. Torres,
On the H\"ormander classes of bilinear pseudodifferential operators II.,
Indiana Univ. Math. J. {\bf 62} (2013), 1733--1764.
	
\bibitem[BT]{BT}
\'A. B\'enyi and R.H. Torres,
Symbolic calculus and the transposes of bilinear pseudodifferential operators,
Comm. Partial Differential Equations {\bf 28} (2003), 1161--1181. 

\bibitem[CM]{CM}
R. Coifman and Y. Meyer,
{Au del\`a des op\'erateurs pseudo-diff\'erentiels},
Ast\'erisque {\bf 57} (1978), 1--185.

\bibitem[DS]{DS}
D. Dos Santos Ferreira and W. Staubach, 
{Global and local regularity for Fourier integral operators on
weighted and unweighted spaces}, 
Mem. Amer. Math. Soc. {\bf 229} (2014).

\bibitem[GP]{GP}
L. Grafakos and M. M. Peloso, 
Bilinear Fourier integral operators, 
J. Pseudo-Differ. Oper. Appl. {\bf 1} (2010), 161--182.

\bibitem[GT]{GT}
L. Grafakos and R. Torres,
{Multilinear Calder\'on-Zygmund theory},
Adv. Math. {\bf 165} (2002), 124--164.

\bibitem[Gol]{Goldberg} 
D. Goldberg, 
A local version of real Hardy spaces, 
Duke Math. J. {\bf 46} (1979), 27--42.

\bibitem[H\"or]{Hor1971}
L. H\"ormander, 
Fourier integral operators. I, 
Acta Math. {\bf 127} (1971), 79--183.

\bibitem[KMT]{KMT}
T. Kato, A. Miyachi, and N. Tomita,
Estimates for some bilinear wave operators,
arXiv:2305.17870.

\bibitem[M]{M-wave}
A. Miyachi, 
On some estimates 
for wave equations in $L^p$ and in $H^p$, 
J.\ Fac.\ Sci. Univ.\ Tokyo, Sect.\ IA Math.\ {\bf 27} 
(1980), 331--354.

\bibitem[RRS1]{RRS-AM}
S. Rodr\'iguez-L\'opez, D. Rule, and W. Staubach,
A Seeger--Sogge--Stein theorem for bilinear Fourier 
integral operators, 
Adv. Math. {\bf 264} (2014), 1--54. 

\bibitem[RRS2]{RRS-TAMS}
S. Rodr\'iguez-L\'opez, D. Rule, and W. Staubach,
On the boundedness of certain bilinear oscillatory integral operators,
Trans. Amer. Math. Soc. {\bf 367} (2015), 6971--6995.

\bibitem[RS]{RS}
S. Rodr\'iguez-L\'opez and W. Staubach, 
Estimates for rough Fourier integral and pseudodifferential 
operators and applications to the boundedness of multilinear operators, 
J. Funct. Anal. {\bf 10} (2013), 2356--2385.

\bibitem[SSS]{SSS}
A. Seeger, C. D. Sogge, and E. M. Stein, 
Regularity properties of Fourier integral operators, 
Ann. of Math. {\bf 134} (1991), 231--251.

\bibitem[S]{S}
E. M. Stein, 
{\it Harmonic Analysis, 
Real-Variable Methods, 
Orthogonality, 
and Oscillating Integrals}, 
Princeton Univ.\ Press, 
Princeton, NJ, 1993.

\end{thebibliography}
\end{document}